\documentclass[12pt]{article}
\usepackage{adjustbox}
\usepackage{amsmath} 
\usepackage{amsthm} 
\usepackage{amssymb} 
\usepackage[dvipsnames]{xcolor} 
\usepackage{mathtools} 
\usepackage{mma}
\usepackage{tikz,tikz-cd}
\usepackage[scr=euler]{mathalfa} 
\usepackage[backref=page,linktocpage]{hyperref} 
\usepackage{graphics,graphicx} 
\usepackage{cleveref}
\usepackage{enumerate} 
\usepackage{float}
\usepackage{setspace}
\usepackage{epic}
\usepackage[latin2]{inputenc}
\usepackage[T1]{fontenc}
\usepackage{hyperref}
\usepackage{enumerate}
\usepackage{caption}
\usetikzlibrary{calc}
\usetikzlibrary{positioning}
\usetikzlibrary{arrows}
\usetikzlibrary{calc,positioning,arrows,decorations.pathreplacing}

\DeclareMathAlphabet{\mathsf}{OT1}{\sfdefault}{m}{n} 
\usepackage[margin=1.20in]{geometry} 

\hypersetup{ 
	colorlinks = true,
	linkbordercolor = {white},
	linkcolor = {BrickRed},
	anchorcolor = {black},
	citecolor = {BrickRed},
	filecolor = {cyan},
	menucolor = {BrickRed},
	runcolor = {cyan},
	urlcolor = {black}
}

\newtheoremstyle{teoremas} 
{11pt}
{11pt}
{\itshape}
{}
{\bfseries}
{}
{.5em}
{}

\theoremstyle{teoremas} 
\newtheorem{theorem}{Theorem}[section] 
\newtheorem{lemma}[theorem]{Lemma} 
\newtheorem{proposition}[theorem]{Proposition} 

\newtheoremstyle{definition} 
{11pt}
{11pt}
{}
{}
{\bfseries}
{}
{.5em}
{}

\theoremstyle{definition} 

\crefname{theorem}{theorem}{theorems} 
\Crefname{theorem}{Theorem}{Theorems} 
\crefname{lemma}{lemma}{lemmas} 
\Crefname{lemma}{Lemma}{Lemmas} 
\crefname{proposition}{proposition}{propositions} 
\Crefname{proposition}{Proposition}{Propositions} 

\DeclareMathOperator{\rk}{rk}
\newcommand{\s}{\mathcal{S}}

\newcommand{\M}{\mathsf{M}}

\makeatletter
\def\dual#1{\expandafter\dual@aux#1\@nil}
\def\dual@aux#1/#2\@nil{\begin{tabular}{@{}c@{}}#1\\#2\end{tabular}}
\makeatother

\begin{document}

\begin{center}
{\large \bf Inverse Kazhdan--Lusztig polynomials of fan matroids}
\end{center}

\begin{center}
 Alice L.L. Gao$^{1}$ , Ya-Xing Li$^{2}$, and Yun Li$^{3}$ \\[6pt]

 $^{1,2,3}$School of Mathematics and Statistics,\\
 Northwestern Polytechnical University, Xi'an, Shaanxi 710072, P.R. China
 
$^{1,2,3}$Shenzhen Research Institute of Northwestern Polytechnical University,\\
Sanhang Science \& Technology Building, No. 45th, Gaoxin South 9th Road, Nanshan District,
Shenzhen City, 518057, P.R. China

Email: $^{1}${\tt llgao@nwpu.edu.cn},
       $^{2}${\tt liya@mail.nwpu.edu.cn},
       $^{3}${\tt liyun091402@163.com}
\end{center}

\noindent\textbf{Abstract.}
The inverse Kazhdan--Lusztig polynomial of a matroid was introduced by Gao and Xie, and the inverse $Z$-polynomial of a matroid was introduced by Ferroni, Matherne, Stevens, and Vecchi.
In this paper, we study these two polynomials for fan matroids, a family of graphic matroids associated with fan graphs. We first derive the generating functions for the inverse Kazhdan--Lusztig polynomials of fan matroids using their recursive definition, 
and then deduce the explicit formulas of these polynomials therefrom. For the inverse $Z$-polynomials of fan matroids, we obtain their generating functions using a parallel generating function approach, and further derive their explicit expansions based on these generating functions.
Additionally, we provide alternative proofs for the above generating functions using the deletion formulas for inverse Kazhdan--Lusztig and inverse $Z$-polynomials. As an application of the explicit formula for inverse Kazhdan--Lusztig polynomials, we prove that the coefficients of the inverse Kazhdan--Lusztig polynomial of the fan matroid  form a log-concave sequence with no internal zeros.

\noindent \emph{AMS Classification 2020:} 
05B35, 05A15, 05E05.


\noindent \emph{Corresponding Author:} Yun Li, liyun091402@163.com

\section{Introduction}

The Kazhdan--Lusztig polynomial of a matroid was first  introduced  by Elias, Proudfoot,
and Wakefield \cite{elias2016kazhdan}.  Building on Kazhdan--Lusztig--Stanley theory for locally finite posets, Gao and Xie \cite{alice2020inverseKL} defined the inverse Kazhdan--Lusztig polynomial $Q_{\M}(t)$ for arbitrary matroid $\M$, 
and conjectured that the coefficients of $Q_{\M}(t)$ are nonnegative.
Braden, Huh, Matherne, Proudfoot, and Wang~\cite{braden2020singular} showed that the coefficient of $t^i$ in $Q_{\M}(t)$ equals the multiplicity of the trivial
module in the degree $\rk(\M)-2i$ piece of the Rouquier complex of $\M$. This cohomological interpretation implies the nonnegativity of the coefficients of $Q_{\M}(t)$ for  all matroids.
Ardila and Sanchez \cite{ardila2023valuations}  proved that
$Q_{\M}(t)$ is
a valuative invariant.
Beyond these  results, explicit formulas for
$Q_{\M}(t)$ have been  derived for several important classes of
matroids.
Ferroni, Nasr, and Vecchi~\cite{ferroni2023stressed} obtained explicit  formulas for the inverse Kazhdan--Lusztig polynomials of paving matroids, and Ferroni and Schr\"{o}ter~\cite{ferroni2024valuative} later extended these results to elementary split matroids, a class that includes paving and uniform matroids.

Proudfoot, Xu, and Young \cite{PXY_2018ELC} introduced the $Z$-polynomial  of a matroid $\M$, which Braden, Huh, Matherne, Proudfoot, and Wang \cite{braden2020singular} showed to coincide with the Hilbert series of the intersection cohomology module of $\M$.
Ferroni, Matherne, Stevens, and Vecchi \cite{ferroni2024hilbert} later defined the inverse $Z$-polynomial $Y_{\M}(t)$  for arbitrary  matroid $\M$. Gao, Ruan, and Xie \cite{gao2025inverse} further studied its fundamental  properties.
Braden, Ferroni, Matherne, and Nepal \cite{braden2025deletion} derived a deletion formula for both the
invariants $Q_{\M}(t)$ and $Y_{\M}(t)$.
The main objective of this paper is to compute the inverse Kazhdan--Lusztig polynomials and the inverse $Z$-polynomials of
fan matroids, a class of graphic matroids associated with fan graphs. 
We also study the log-concavity of the inverse Kazhdan--Lusztig polynomials for such  matroids.

Let us first recall basic notions and notation for graphic matroids, mainly 
following Lu, Xie, and Yang~\cite{xie2018}.
Let $G$ be a loopless graph with vertex set $V(G)$ and edge set $E(G)$,
and let $\M(G)$ be its associated graphic matroid.
The ground set of $\M(G)$ is $E(G)$,  with independent sets precisely the forests of $G$.
For any subset $A\subseteq E(G)$, let $G[A]$ be the subgraph of $G$ with edge set $A$ and vertex set consisting of all endpoints of edges in $A$.
The rank of $A$ is  given by
$$
\rk(A)  = |V(G[A])| - c(G[A]),
$$
where $c(G[A])$  denotes the number of connected components of $G[A]$.
The rank of the graphic matroid $\M(G)$ is the rank of the edge set $E(G)$, i.e., $|V(G)|$ minus the number of connected components of $G$. We denote this rank by $\rk(G)$.
A flat of $\M(G)$ is either $E(G)$
or a proper subset $F\subseteq E(G)$ whose rank increases strictly
when any $e\in E(G)\setminus F$ is added.
Flats of $\M(G)$ correspond bijectively to partitions of $V(G)$
into vertex sets of connected induced subgraphs,
called compositions of $G$; see \cite{knopfmacher2001graph}.
We denote by $\mathcal{C}(G)$ the set of all compositions of $G$,
and for $C\in\mathcal{C}(G)$, let $|C|$ be the number of parts of $C$.
For a flat $F$ of $\M(G)$,
the restriction and contraction of $\M(G)$ to $F$
are naturally identified with graphic matroids via
\begin{align*}
\M(G)|_F = \M(G[F]), \qquad \M(G)/F = \M(G/F),
\end{align*}
where $G/F$ denotes the graph obtained by contracting all edges in $F$ (in any order).
If a composition $C \in \mathcal{C}(G)$ corresponds to a flat $F$,
we also write $G[C]$ for $G[F]$ and $G/C$ for $G/F$; see \cite[p. 61 and p. 63]{welsh1976matroid}.

We now introduce the inverse Kazhdan--Lusztig polynomial for graphic matroids. 
For notational convenience, we write $Q_G(t)$ instead of $Q_{\M(G)}(t)$ and refer to it as the inverse Kazhdan--Lusztig polynomial of a given graph $G$.
This polynomial is uniquely determined by the following conditions
\begin{itemize}
    \item If $\rk(G) = 0$, i.e., $E(G) = \emptyset$, then $Q_G(t) = 1$.
    \item If $\rk(G)> 0$, then $\deg Q_G(t) < \frac{1}{2}\rk(G)$, and
    \begin{align}\label{ikl-reform}
    (-t)^{\rk(G)}Q_G(t^{-1}) = \sum_{C\in \mathcal{C}(G)} (-1)^{\rk(G[C])} Q_{G[C]}(t) \cdot t^{|C|} \chi_{G/C}(t^{-1}),
\end{align}
where $\chi_G(t)$ is the chromatic polynomial of $G$.
\end{itemize}
We note that the original definition of the inverse Kazhdan--Lusztig
polynomial is formulated in terms of the characteristic polynomial
$\chi_{\M(G)}(t)$ of the matroid $\M(G)$. Since
\begin{align}\label{chim-chi}
\chi_{\M(G)}(t)=t^{-c(G)}\chi_G(t),
\end{align}
where $c(G)$ is the number of connected components of $G$
\cite[p.~262]{welsh1976matroid}, we may equivalently work with the
chromatic polynomial $\chi_G(t)$.
Besides,
if $G$ has parallel edges,
let $\mathrm{si}(G)$ denote the underlying simple graph of $G$ obtained by deleting all loops, if any, and replacing each parallel class by a single edge. Since $\mathcal{C}(G)=\mathcal{C}(\mathrm{si}(G))$, the lattices of flats of $\M(G)$ and $\M(\mathrm{si}(G))$ are isomorphic. Consequently, $Q_G(t)=Q_{\mathrm{si}(G)}(t)$.

A fan graph $F_n$ ($n\geq 1$) is  a graph on $n+1$ vertices, which is obtained by connecting a distinguished vertex to every vertex of a path with $n$ vertices. 
We call $\M(F_n)$
the fan matroid. 
Our first main result provides
an explicit  formula for the inverse Kazhdan--Lusztig polynomial  of the fan matroid.

\begin{theorem}\label{thm-fan}
For any fan matroid $F_n$ with $n\geq1$, 
\begin{align}\label{equ-thm-fan}
Q_{F_{n}}(t)= \sum_{k=0}^{\lfloor\frac{n-1}{2}\rfloor}\frac{(n-2k)2^{n-2k-1}}{n}\binom{n}{k}t^{k}.
\end{align}
\end{theorem}

Recall that a polynomial $f(t)=\sum_{i=0}^{n}a_{i}t^{i}$
with real coefficients is  called log-concave if
$$a_{i}^{2} \ge a_{i-1}a_{i+1} \qquad\text{~for~all~} 1 \le i \le n-1,$$
and it is said to have no internal zeros if there do not exist
indices $0 \le i < j < k \le n$ with $a_i, a_k \neq 0$ and $a_j=0$. 
Motivated by the log-concavity conjecture for matroid Kazhdan--Lusztig polynomials
due to Elias, Proudfoot, and Wakefield~\cite{elias2016kazhdan},
Gao and Xie~\cite{alice2020inverseKL} proposed an analogous conjecture
for the inverse Kazhdan--Lusztig polynomial $Q_\M(t)$.
They conjectured that, for every matroid $\M$, the coefficients of $Q_\M(t)$
form a log-concave sequence with no internal zeros.
This conjecture has been verified for several well-studied  classes of matroids:
Gao and Xie~\cite{alice2020inverseKL} for uniform matroids,
Xie and Zhang~\cite{XieZhang2025Logconcavity} for paving matroids,
and Gao, Li, and Xie~\cite{gao2025equivariant} for thagomizer matroids.
In this paper, we extend these results to fan matroids.

\begin{theorem}\label{thm-fan-log-concave}
For any positive integer $n$, the coefficients of  $Q_{F_n}(t)$ form a log-concave sequence
with no internal zeros.  
\end{theorem}

We now study the inverse $Z$-polynomials for fan matroids. 
For notational convenience, we also write $Y_G(t)$ instead of $Y_{\M(G)}(t)$ and refer to it as the inverse $Z$-polynomial of the graph $G$.
It is defined by
\begin{align}\label{invz-def}
    Y_{G}(t):=(-1)^{\rk(G)}\sum_{C\in \mathcal{C}(G)} (-1)^{\rk(G[C])} Q_{G[C]}(t) \cdot t^{\rk (G/C)} \mu_{G/C},
\end{align}
where $\mu_{G}$ denotes the M\"obius invariant of $\M(G)$. 
Since the definition of $Y_G(t)$ also depends only on the lattice of flats, it follows that  $Y_G(t)=Y_{\mathrm{si}(G)}(t)$.
For fan
matroids, we obtain the following explicit formula for their inverse $Z$-polynomials.

\begin{theorem} \label{thm-fan-invz}
For any fan matroid $F_n$ with $n\ge 1$, $Y_{F_n}(t)$ is a polynomial of degree $n$, and for each integer $k$ with $0\le k\le n$, the coefficient of $t^k$ is given by 
\begin{align}\label{thm-fan-invz-equation}
d_{n,k}=
2^{n-1}\binom{n}{k}+
\sum_{j=1}^{\lfloor n/2\rfloor}\sum_{r=1}^{j}2^{n-2j-1} \frac{(-1)^{r} r (2r+n-2j)}{(2j-r)(r+n-2j)}\binom{n-2j}{k-j}\binom{2j-r}{j}\binom{r+n-2j}{r}.
\end{align}
\end{theorem}

This paper is organized as follows.
In Section~\ref{sec-thm1.1}, 
we first derive the generating function for the inverse Kazhdan--Lusztig polynomials of fan matroids via their recursive definition, and then deduce an explicit  formula
for these polynomials.
In Section \ref{In_Z_Fan},
we first derive the generating function for the inverse $Z$-polynomials of fan matroids using a parallel generating function approach, and further obtain their explicit expansion from the generating function.
In Section \ref{dele-section},
we provide alternative proofs of the above generating functions via the deletion formulas for inverse Kazhdan--Lusztig and inverse $Z$-polynomials.
In Section~\ref{sec-thm1.3}, we prove that the coefficients 
of $Q_{F_n}(t)$ form a log-concave sequence with no internal zeros.



\section{Inverse Kazhdan--Lusztig polynomials}\label{sec-thm1.1}

In this section we prove Theorem~\ref{thm-fan}, which gives an explicit
formula for the inverse Kazhdan--Lusztig polynomials of fan matroids.
Our approach follows the generating-function method of
Lu, Xie, and Yang~\cite{xie2018} for Kazhdan--Lusztig polynomials of fan
matroids.

\subsection{Generating functions}
Lu, Xie, and Yang~\cite{xie2018} computed the generating functions for the
Kazhdan--Lusztig polynomials of fan matroids $P_{F_n}(t)$.  We establish the analogous
generating function for the inverse Kazhdan--Lusztig polynomials
$Q_{F_n}(t)$.

Let
\begin{align}\label{fan-pro1}
    \Psi(t,u) := \sum_{n=0}^{\infty} Q_{F_n}(t) \, u^n,
\end{align}
where $F_0$ is the single-vertex graph. Our main result in this
subsection is as follows.

\begin{theorem}\label{thm-fan-pro}
We have
\begin{align}\label{fan-pro2}
    \Psi(t,u) = 1 + \frac{1 - 4u - \sqrt{1 - 4 t u^2}}{2(-2 + 4u + t u)}.
\end{align}
\end{theorem}

Following \cite{xie2018}, we start from the recursive definition
\eqref{ikl-reform}.
After multiplying by $u^n$ and summing over $n$,
this recursion becomes a functional equation for the generating function
$\Psi(t,u)$.
We derive
this functional equation first.
Applying \eqref{ikl-reform} to  the fan graph $F_n$ gives
\begin{align*}
(-t)^{\rk(F_n)} Q_{F_n}(t^{-1})
= \sum_{C \in \mathcal{C}(F_n)}
(-1)^{\rk(F_n[C])} 
    Q_{F_n[C]}(t)  \cdot t^{|C|} 
    \chi_{F_n/C}(t^{-1}).
\end{align*}
Multiplying both sides by $u^n$ and  summing over all $n \ge 0$ yields
\begin{align}\label{eq:genfun-sum}
\sum_{n=0}^{\infty} \left((-t)^{\rk(F_n)} Q_{F_n}(t^{-1})\right) u^n
=
\sum_{n=0}^{\infty} \left( \sum_{C \in \mathcal{C}(F_n)}
   (-1)^{\rk(F_n[C])} Q_{F_n[C]}(t) 
   \cdot 
   t^{|C|} \chi_{F_n/C}(t^{-1}) \right) u^n.
\end{align}
Since $\rk(F_n)=n$,
the left-hand side of \eqref{eq:genfun-sum} becomes
\begin{align*}
\sum_{n=0}^{\infty} Q_{F_n}(t^{-1})  (-tu)^n 
= \Psi(t^{-1}, -tu),
\end{align*}
and hence
\begin{align}\label{fan-pro3}
\Psi(t^{-1}, -tu)
= \sum_{n=0}^{\infty} 
    \left(
      \sum_{C \in \mathcal{C}(F_n)}
    (-1)^{\rk(F_n[C])} 
        Q_{F_n[C]}(t)  \cdot 
        t^{|C|} \chi_{F_n/C}(t^{-1})
    \right) u^n.
\end{align}
If the right-hand side of
\eqref{fan-pro3}
can be expressed in terms of
$\Psi(t,u)$,
then we obtain the desired
functional equation satisfied by
$\Psi(t,u)$.
Since the inverse Kazhdan--Lusztig polynomial of a matroid is uniquely determined by the recurrence relation and degree bound.
Once this equation is obtained, it suffices to check that the right-hand
side of \eqref{fan-pro2} satisfies it and has the required degree bound.

To this end,
we recall the description of the flats of the fan graph $F_n$
given in \cite{xie2018},
and summarize the combinatorial
constructions
used later in the proof.
A weak composition of  $n$ is a  sequence 
$(a_1,a_2,\ldots,a_k)$ of non-negative integers such that 
$a_1+a_2+\cdots+a_k=n$. 
Let $\mathcal{S}_n$
denote the
set of compositions with all parts strictly positive.  By convention, we set
$\mathcal{S}_0=\{()\}$.  
Let $\mathcal{E}_n$ denote the set of weak compositions of $n$ with  even number of parts, say $(a_1,\dots,a_{2k})$, such that $a_i \ge 1$ for $1<i<2k$. We note that $\mathcal{E}_0=\{(0,0)\}$.  
For each $\sigma=(a_1,\dots,a_{2k}) \in \mathcal{E}_n$, define
$$
\theta(\sigma) := \Bigl\{ (A_1,\dots,A_{2k}) \;\Big|\; 
A_{2i-1}=(a_{2i-1}) \text{~and~} A_{2i}\in \mathcal{S}_{a_{2i}} \text{for~} 1\le i\le k. \Bigr\}
$$
We then set
$$
\mathcal{C}'_n: = \bigcup_{\sigma \in \mathcal{E}_n} \theta(\sigma).
$$

Lu, Xie, and Yang \cite[Lemma~4.2]{xie2018} proved that for each $n\geq 1$, there exists a bijection
$$\phi :
\mathcal{C}(F_n) \longrightarrow \mathcal{C}'_n .$$
The same conclusion clearly holds for $n=0$.
We briefly recall the construction of $\phi$. 
The fan graph $F_n$ is drawn in the plane
such that vertex $0$ is adjacent to  the path
with vertices 
$1,2,\dots,n$
ordered from left to right.
Given a  composition $C \in \mathcal{C}(F_n)$, one considers the unique connected component
of the induced subgraph that contains $0$.
Removing the vertex $0$ from this component
produces a sequence of subpaths
$T_1,T_2,\ldots,T_k$,
ordered from left to right along the path
$1,2,\ldots,n$.
These subpaths partition the remaining
vertices into $k+1$ consecutive segments,
denoted by
$S_1,S_2,\ldots,S_{k+1}$.
Each segment $S_i$ is a forest,
that is, a disjoint union of paths,
and therefore determines a composition
$A_i$ of $|V(S_i)|$
given by the sizes of its connected components
in left-to-right order.
Note that $S_1$ and $S_{k+1}$
may be empty,
whereas $S_i$ is nonempty for
$2\le i\le k$.

For illustration,
consider the compositions
$C_1 = \{\{0,1,2,7\}, \{3,4\}, \{5\}, \{6\}, \{8\}\}$ and 
$C_2 = \{\{0,3,4,7,8\}, \{1,2\}, \{5\}, \{6\}\}$ of $F_8$, 
the corresponding  subpaths $T_i$
and segments $S_i$
are shown in
Figures~\ref{S1empty} and~\ref{S1nonempty}.

\begin{figure}[H]
\centering
\begin{minipage}[b]{0.48\textwidth}
\centering
\begin{tikzpicture}[line cap=round,line join=round,scale=0.7] 
\draw (0,0) node[circle,fill,inner sep=0pt,label=above:0] (0) {};

\foreach \s in {1,2,...,8}   
  \draw (\s-4.5,-3) node[circle,fill,inner sep=0.5pt,label=below:\s] (\s) {};

\draw[color=blue,line width=1.5pt] (-3.5,-3)--(-2.5,-3)--(0,0)--cycle;
\draw[color=blue,line width=1.5pt] (0,0)--(2.5,-3);
\draw[color=blue,line width=1.5pt] (-1.5,-3)--(-0.5,-3);

\foreach \s in {3,4,...,6}  \draw (0)--(\s);
\foreach \s in {8}          \draw (0)--(\s);

\foreach \s in {1,...,7} \draw (\s-4.5,-3)--(\s-3.5,-3);

\foreach \s in {1,2,...,8}   
  \draw (\s-4.5,-3.8) node[circle,fill,inner sep=0.5pt,label=below:] (\s) {};

\draw[color=blue,line width=1.5pt] (1)--(2);
\node[draw=none,minimum size=3mm,inner sep=0pt,yshift=-8pt,label=below:$S_1$]   at  (-4.5,-3.7)   {};
\node[draw=none,minimum size=3mm,inner sep=0pt,yshift=-8pt,label=below:$T_1$]   at  (-3,-3.7)   {};
\draw [decorate,decoration={brace,amplitude=5pt},xshift=-0.0cm,yshift=-4pt]
(1.5,-3.8) -- (-1.5,-3.8) node [black,midway,xshift=-0pt,yshift=-15pt]  {$S_2$};
\draw[color=blue,line width=1.5pt] (3)--(4);
\node[draw=none,minimum size=3mm,inner sep=0pt,yshift=-8pt,label=below:$T_2$]   at  (2.5,-3.7)   {};
\node[draw=none,minimum size=3mm,inner sep=0pt,yshift=-8pt,label=below:$S_3$]   at  (3.5,-3.7)   {};
\end{tikzpicture}
\captionsetup{font=footnotesize}
\caption{Constructions of $T_i's$ and  $S_i's$ for $C_1$.}
\label{S1empty}
\end{minipage}
\hfill
\begin{minipage}[b]{0.48\textwidth}
\centering
\begin{tikzpicture}[line cap=round,line join=round,scale=0.7]
\draw (0,0) node[circle,fill,inner sep=0pt,label=above:0] (0) {};

\foreach \s in {1,2,...,8}   
  \draw (\s-4.5,-3) node[circle,fill,inner sep=0.5pt,label=below:\s] (\s) {};

\draw[color=blue,line width=1.5pt] (-1.5,-3)--(-0.5,-3)--(0,0)--cycle;
\draw[color=blue,line width=1.5pt] (2.5,-3)--(3.5,-3)--(0,0)--cycle;
\draw[color=blue,line width=1.5pt] (-3.5,-3)--(-2.5,-3);

\foreach \s in {1,2}  \draw (0)--(\s);
\foreach \s in {5,6}          \draw (0)--(\s);

\foreach \s in {1,...,7} \draw (\s-4.5,-3)--(\s-3.5,-3);

\foreach \s in {1,2,...,8}   
  \draw (\s-4.5,-3.8) node[circle,fill,inner sep=0.5pt,label=below:] (\s) {};

\draw[color=blue,line width=1.5pt] (1)--(2);
\node[draw=none,minimum size=3mm,inner sep=0pt,yshift=-8pt,label=below:$S_1$]   at  (-3,-3.7)   {};
\draw[color=blue,line width=1.5pt] (3)--(4);
\node[draw=none,minimum size=3mm,inner sep=0pt,yshift=-8pt,label=below:$T_1$]   at  (-1,-3.7)   {};
\draw [decorate,decoration={brace,amplitude=5pt},xshift=-0.0cm,yshift=-4pt]
(1.5,-3.8) -- (0.5,-3.8) node [black,midway,xshift=-0pt,yshift=-15pt]  {$S_2$};
\draw[color=blue,line width=1.5pt] (7)--(8);
\node[draw=none,minimum size=3mm,inner sep=0pt,yshift=-8pt,label=below:$T_2$]   at  (3,-3.7)   {};
\node[draw=none,minimum size=3mm,inner sep=0pt,yshift=-8pt,label=below:$S_3$]   at  (4.5,-3.7)   {};
\end{tikzpicture}
\captionsetup{font=footnotesize}
\caption{Constructions of $T_i's$ and  $S_i's$ for $C_2$.}
\label{S1nonempty}
\end{minipage}
\end{figure}
With the above notation,
the bijection $\phi$ is defined by
\begin{align}\label{phiC}
\phi(C)
=
\begin{cases}
\bigl(
(|V(T_1)|),A_2,(|V(T_2)|),A_3,\dots,
A_k,(|V(T_k)|),A_{k+1}
\bigr),
& \text{if } S_1=\varnothing,\\[1ex]
\bigl(
(0),A_1,(|V(T_1)|),A_2,\dots,
A_k,(|V(T_k)|),A_{k+1}
\bigr),
& \text{if } S_1\neq\varnothing.
\end{cases}
\end{align}
Using this bijection,
we  show that for any
$C\in\mathcal{C}(F_n)$,
the summand
$$
(-1)^{\rk(F_n[C])} Q_{F_n[C]}(t)\cdot t^{|C|}\chi_{F_n/C}(t^{-1})
$$
can be evaluated by assigning a suitable weight
to the corresponding element
$\phi(C) \in \mathcal{C}'_n$. 

Given
$
A = (A_1, A_2, \dots, A_{2k-1}, A_{2k}) \in \mathcal{C}'_n
$ with $A_{2i-1} = (a_{2i-1})$ and 
$A_{2i} = (b_{i1}, \dots, b_{i\ell_i}) \in \mathcal{S}_{a_{2i}}$ for $1\leq i \leq k$,  
define  the weight of $A$ by
\begin{align}\label{eq-weight-function}
w(A)
  := \prod_{i=1}^{k} 
         (-1)^{\ell_i}    t^{\ell_i+1}Q_{F_{a_{2i-1}}}(t) \chi_{F_{\ell_i}}(t^{-1}).
\end{align}
We have the following result.

\begin{lemma}\label{wterm}
For any $C \in \mathcal{C}(F_n)$, 
\begin{align}
     (-1)^{\rk( F_n[C])}Q_{F_n[C]}(t)\cdot t^{|C|} \chi_{F_n/C}(t^{-1}) = (-1)^{n}w\left(\phi(C)\right).
\end{align}
\end{lemma}

\begin{proof}
Fix $C \in \mathcal{C}(F_n)$. Write
$$
\phi(C) = (A_1, A_2, \ldots, A_{2k-1}, A_{2k}),
$$
where $A_{2i-1} = (a_{2i-1})$ and 
$A_{2i} = (b_{i1}, b_{i2}, \ldots, b_{i\ell_i}) \in \mathcal{S}_{a_{2i}}$ 
for $1 \le i \le k$. 
The induced subgraph $F_n[C]$ consists of a unique component 
containing the vertex $0$, obtained by identifying the vertex $0$ 
of the fan graphs $F_{a_1}, \dots, F_{a_{2k-1}}$, 
together with the disjoint union of
paths  $H_{b_{ij}}$ for $1 \le i \le k$ 
and $1 \le j \le \ell_i$.
It follows that 
\begin{align}\label{rkFnc}
    \rk(F_n[C]) = n+1 - |C|=n-(\ell_1 + \ell_2 + \cdots + \ell_k).
\end{align}
Since inverse Kazhdan--Lusztig polynomials
are multiplicative with respect to
direct sums of matroids, we obtain
\begin{align}\label{QFnc}
Q_{F_{n}[C]}(t)=\prod_{i=1}^{k} \left(Q_{F_{a_{2i-1}}}(t)  \prod_{j=1}^{\ell_i} Q_{H_{b_{ij}}}(t)\right)
=\prod_{i=1}^{k}Q_{F_{a_{2i-1}}}(t),
\end{align}
where the last equation comes from the fact that $Q_{H_b}(t)=1$ for any path $H_b$.

We now compute $\chi_{F_{n}/C}(t^{-1})$. 
Note that even though $G/C$ may have multiple edges, its chromatic polynomial agrees with that of its simplification. We therefore identify them when computing the chromatic polynomial. 
By contracting each block of $C$
to a single vertex,
the graph $F_n/C$ is isomorphic to the graph
obtained by identifying the distinguished
vertices of the fan graphs
$F_{\ell_1},\ldots,F_{\ell_k}$.
Recall that if  a graph $G$ contains $m$ biconnected components and $c(G)$ connected components,      then its chromatic polynomial
satisfies
\begin{align}\label{multi-chrpol}
\chi_G(t)=t^{-(m-c(G))}\prod_{i=1}^{m}\chi_{G_i}(t),
\end{align}
as shown in \cite[Lemma~2.1]{xie2018}.
Applying  \eqref{multi-chrpol} to $F_n/C$, we obtain
\begin{align}\label{chi-Fcon-C}
    \chi_{F_n/C}(t) = t^{-(k-1)}\prod_{i=1}^{k}{\chi_{F_{\ell_i}}(t)},
\end{align}  
and hence
$$
\chi_{F_n/C}(t^{-1}) = t^{k-1}\prod_{i=1}^{k}{\chi_{F_{\ell_i}}(t^{-1})}.
$$
Combining the above identities yields the desired result.
\end{proof}

Using the bijection $\phi: \mathcal{C}(F_n) \to \mathcal{C}_n'$ and Lemma~\ref{wterm}, equation~\eqref{fan-pro3} can be rewritten in terms of the weight function $w(A)$ as
$$
\Psi(t^{-1}, -tu) = \sum_{n=0}^{\infty} \left(\sum_{A \in \mathcal{C}_n'} w(A)\right) \, (-u)^n.
$$
Define
$$
\Phi(u) := \sum_{n=0}^{\infty} \left(\sum_{A \in \mathcal{C}_n'} w(A)\right) \, u^n.
$$
Then we have
\begin{align}\label{eq-key-gf}
\Psi(t^{-1}, -tu) = \Phi(-u).
\end{align}
To derive the functional equation satisfied by  $\Psi(t,u)$, 
it  suffices to express $\Phi(u)$ in terms of $\Psi(t,u)$ 
via generating function methods.

We adopt the combinatorial structures introduced by 
Lu, Xie, and Yang \cite{xie2018}, with a minor modification 
of the associated weights. 
Each element
$A = (A_1, A_2, \dots, A_{2k-1}, A_{2k}) \in \mathcal{C}_n'$
may be viewed as a combinatorial structure of type 
$\mathcal{A}$ on an interval of size $n$. 
For the purpose of generating functions, let 
$\mathcal{A}^o$ and $\mathcal{A}^e$ denote two types of structures, 
corresponding respectively to the components of $A$ 
of odd indices and of even indices.
Precisely, a type $\mathcal{A}^o$ structure on an interval of size $n$ 
corresponds to the weak composition $(n)$, 
with weight 
$$
w^o\left((n)\right): = t \, Q_{F_n}(t).
$$
A type $\mathcal{A}^e$ structure on an interval of size $n$ 
corresponds to a composition $(b_1, \dots, b_k) \in \mathcal{S}_n$, with weight 
$$
    w^e\left((b_1, \ldots, b_k)\right) := (-t)^k\chi_{F_k}(t^{-1}).
$$
Note that the unique type~$\mathcal{A}^o$ structure of size 0 is $(0)$,  with weight $t$, and the  unique type~$\mathcal{A}^e$ structure of size $0$ is the empty composition $(\,)$, 
 with weight $1/t$.

For each $n \ge 0$, let $\mathcal{A}^o_n$ (resp.~$\mathcal{A}^e_n$)  denote the set of all type~$\mathcal{A}^o$ (resp.~type~$\mathcal{A}^e$) structures  on an interval of size $n$. Define the corresponding generating functions by
\begin{align}
\Phi^o(u): &= \sum_{n=1}^{\infty} \left( \sum_{A^o \in \mathcal{A}^o_n} w^o(A^o) \right) u^n, \label{gf-def-1} \\
\Phi^e(u): &= \sum_{n=1}^{\infty} \left( \sum_{A^e \in \mathcal{A}^e_n} w^e(A^e) \right) u^n. \label{gf-def-2}
\end{align}
The following lemma provides explicit expressions 
for these generating functions.

\begin{lemma}\label{gen-o and e}
We have
\begin{align}
\Phi^o(u) &= t \left( \Psi(t,u) - 1 \right), \label{gf-1} \vspace{4mm}\\
\Phi^e(u) &= \frac{(1-t) u}{t(2 t u - 1)}. \label{gf-2}
\end{align}
\end{lemma}

\begin{proof}
We first establish~\eqref{gf-1}.
By definition of $\Psi(t,u)$, 
$$
    \Phi^o(u) = \sum_{n=1}^{\infty} t  Q_{F_n}(t) \, u^n = t \bigl( \Psi(t,u) - 1 \bigr).
$$

We now prove~\eqref{gf-2}.
A composition $(b_1,\ldots,b_k)\in\mathcal{S}_n$
may be viewed as a sequence of $k$ nonempty
intervals whose lengths sum to $n$.
Assigning weight $-t$ to each interval and
weight $\chi_{F_k}(t^{-1})$ to the sequence,
we obtain
$$
\sum_{b_j=1}^{\infty} (-t)u^{b_j}
=
\frac{tu}{u-1}
\quad \text{and} \quad
\sum_{k=1}^{\infty}
\chi_{F_k}(t^{-1})u^k 
=
\frac{(1-t)u}{t(t-u+2tu)}.
$$
Applying the composition formula for
generating functions
\cite[Proposition~3.3]{xie2018},
we obtain
$$
    \Phi^{e}(u)
    = \frac{(1 - t) u}{t (2 t u - 1)}.
$$
This completes the proof.
\end{proof}

We continue to use the combinatorial
constructions introduced in
Lu, Xie, and Yang~\cite{xie2018}.
Let $\mathcal{A}^m$ denote the 
combinatorial structure consisting of
a finite alternating sequence  of
$\mathcal{A}^e$ and $\mathcal{A}^o$ structures.
Each such sequence begins with a
structure of type $\mathcal{A}^e$ and
ends with a structure of type
$\mathcal{A}^o$.
Moreover, both the $\mathcal{A}^e$  and
$\mathcal{A}^o$ structures are required
to be nonempty.
The weight function $w^m$ on type $\mathcal{A}^m$
structures  is defined as the product
of the weights of its components.
For each $n \ge 0$, let $\mathcal{A}^m_n$ denote the set of type $\mathcal{A}^m$ structures which can be built on an interval of size $n$, and define the generating
function
\begin{align*}
\Phi^m(u): &= \sum_{n=0}^{\infty} \left( \sum_{A^m \in \mathcal{A}^m_n} w^m(A^m) \right) u^n.
\end{align*}
By the composition formula of generating functions, one obtains
$$
\Phi^m(u) = \sum_{m=0}^{\infty}
\bigl(\Phi^e(u)\Phi^o(u)\bigr)^m
= \frac{1}{1-\Phi^e(u)\Phi^o(u)};
$$
see~\cite[Section~4]{xie2018}.

We are now ready to express $\Phi(u)$ in
terms of $\Psi(t,u)$.

\begin{lemma}\label{invkl-Phi}
We have
\begin{align}\label{eq-another}
\Phi(u)
=
\frac{\left(t+\Phi^o(u)\right)\left(t^{-1}+\Phi^e(u)\right)}{1-\Phi^e(u)\Phi^o(u)}.
\end{align}
\end{lemma}

\begin{proof}
Every structure of type $\mathcal{A}$ 
admits a unique decomposition into
three consecutive parts.
The first part is 
of type $\mathcal{A}^o$, which may
be empty.
It is followed by a structure of type
$\mathcal{A}^m$.
The decomposition ends with a
structure of type $\mathcal{A}^e$,
which may also be empty.
The weight of the entire structure
is the product of the weights of
these three parts.
Translating this decomposition into
generating functions yields
\begin{align*}
\Phi(u)=\left(t+\Phi^o(u)\right) \times 
\Phi^m(u)
\times
\left(\frac{1}{t}+\Phi^e(u)\right).
\end{align*}
Substituting the expression for
$\Phi^m(u)$ gives
\begin{align*}
\Phi(u)=\left(t+\Phi^o(u)\right) \times \frac{1}{1-\Phi^e(u)\Phi^o(u)}
\times \left(\frac{1}{t}+\Phi^e(u)\right).
\end{align*}
This completes the proof.
\end{proof}

We are now in a position  to prove Theorem~\ref{thm-fan-pro}.

\begin{proof}[First proof of Theorem~\ref{thm-fan-pro}.]
By  \eqref{eq-key-gf} and \eqref{eq-another}, the generating function
$\Psi(t,u)$
satisfies the functional equation
\begin{align}\label{eq-gf-klpol-fan-funceqn}
\Psi(t^{-1}, -tu) = 
\frac{(t + \Phi^o(-u)) (t^{-1} + \Phi^e(-u))}{1 - \Phi^e(-u)\Phi^o(-u)}.
\end{align}
We first check that the right-hand side of \eqref{fan-pro2} satisfies the same identity.
This is purely algebraic
and can be verified directly.
For completeness,
we include the Mathematica code below.
{\renewcommand\baselinestretch{2.5}
\begin{mma}
\In \Psi[u\_]:=\frac{1-4u-\sqrt{1-4tu^{2}}}{2(-2+4u+tu)}+1;\\
\In \Phi^{o}[u\_]:=t(\Psi[u]-1);\\
\In \Phi^{e}[u\_]:=\frac{u(1-t)}{t(2tu-1)};\\
\In \Phi^{F}[u\_]:=\frac{(t+\Phi^{o}[-u])(t^{-1}+\Phi^{e}[-u])}{1-\Phi^{e}[-u]\Phi^{o}[-u]};\\
\end{mma}
}
\begin{mma}
\In |Simplify|[(\Psi[u]/.\{t\to t^{-1},u\to -t u\})==\Phi^F[u]]\\
\end{mma}
\begin{mma}
\Out  |True|\\
\end{mma}

We now verify
the required degree bound.
To this end, let $\widehat\Psi(t,u)$ denote the right-hand side of \eqref{fan-pro2}. 
Multiplying both the numerator and denominator by 
$1-4u+\sqrt{1-4tu^2}$ gives
\begin{align}\label{eq-psicap-mult}
\widehat\Psi(t,u)=1 + \frac{2u}{1-4u+\sqrt{1-4tu^2}}.
\end{align}
Let $\mathbf{A}[[t]]$ denote the ring of formal power series in  $t$  over a commutative ring  $\mathbf{A}$ with unity, and let $\mathbf{A}[[t,u]]:=\mathbf{A}[[t]][[u]]$  denote the ring of formal power series in  variables $t$ and $u$  over  $\mathbf{A}$.
A power series $a(t,u)\in \mathbf{A}[[t,u]]$ 
is invertible
if its constant term $a(0,0)$
is invertible in  $\mathbf{A}$; see \cite[Section~2.5]{kauers2011formal}.
Take $\mathbf{A}=\mathbb{Z}$.
Using the binomial expansion,
\begin{equation}\label{genhaoexp}
1-4u+\sqrt{1-4tu^2}
=
2-4u+\sum_{i=1}^{\infty} \binom{1/2}{i}(-4)^i t^i u^{2i}
=2-4u-2\sum_{i=1}^{\infty}C_{i-1}t^i u^{2i},
\end{equation}
where
$C_i=\frac{1}{i+1}\binom{2i}{i}$ denotes
the $i$-th Catalan number.
It follows that
\begin{align*}
  a(t,u):=  \frac{1-4u+\sqrt{1-4tu^2}}{2} \in \mathbb{Z}[[t,u]].
\end{align*}
Since its constant term
$a(0,0)=1$,
the series $a(t,u)$
is invertible in $\mathbb{Z}[[t,u]]$.
Thus
$\widehat\Psi(t,u) =1+2ua(t,u)^{-1}
\in \mathbb{Z}[[t,u]]$.

Define $$\widehat Q_{F_n}(t):=[u^n]\widehat\Psi(t,u) \in \mathbb{Z}[[t]].$$
Substituting $u=0$ into \eqref{eq-psicap-mult} gives
$\widehat Q_{F_0}(t)=1.$
We then obtain
\begin{align}\label{psicap}
\sum_{n=1}^{\infty}\widehat Q_{F_n}(t)  u^n
=\frac{2u}{1-4u+\sqrt{1-4tu^2}}.
\end{align}
Substituting \eqref{genhaoexp} into \eqref{psicap} yields
\begin{align}\label{psicap-2}
  \left(
  \sum_{n=1}^{\infty}\widehat Q_{F_n}(t)  u^n
  \right)  \left(
  1-2u-\sum_{i=1}^{\infty} C_{i-1} t^i u^{2i}
  \right)=u.
\end{align}
Comparing coefficients of $u$ and $u^2$ in \eqref{psicap-2} gives
$\widehat Q_{F_1}(t)=1$ and $\widehat Q_{F_2}(t)=2$. 
For $n\geq 3$,
a comparison of the coefficients of $u^n$ yields the recurrence relation
\begin{align}\label{degree-1}
    \widehat Q_{F_n}(t) = 2\widehat Q_{F_{n-1}}(t) + \sum_{i=1}^{\lfloor (n-1)/2 \rfloor}C_{i-1} t^i \widehat Q_{F_{n-2i}}(t).
\end{align}

We now prove the degree bound by induction on 
$n$.
The statement holds for $n=1,2$.
Assume that for all $2 < k < n$,
\begin{align*}
    \deg \widehat Q_{F_k}(t)\le \Big\lfloor\frac{k-1}{2}\Big\rfloor.
\end{align*}
In particular,
\begin{align*}
    \deg \widehat Q_{F_{n-1}}(t)\le \Big\lfloor\frac{n-2}{2}\Big\rfloor \le \Big\lfloor\frac{n-1}{2}\Big\rfloor.
\end{align*}
Moreover, for each $1\le i\le\lfloor (n-1)/2\rfloor$,
\begin{align*}
    \deg t^i \widehat Q_{F_{n-2i}}(t)\le i + \Big\lfloor\frac{n-2i-1}{2}\Big\rfloor \le \Big\lfloor\frac{n-1}{2}\Big\rfloor.
\end{align*} 
Hence every term on the right-hand side of \eqref{degree-1} has degree at most $\lfloor\frac{n-1}{2}\rfloor$. Therefore
\begin{align*}
    \deg \widehat Q_{F_n}(t)\le \Big\lfloor\frac{n-1}{2}\Big\rfloor.
\end{align*}
This immediately implies
\begin{equation}
\deg \widehat Q_{F_n}(t)<\frac{n}{2} \qquad \text{for~all~} n\geq 1.
\end{equation}
The proof is complete.
\end{proof}

\subsection{Inverse Kazhdan--Lusztig polynomials of fan matroids}\label{sec-thm1.2}

This subsection aims to determine the inverse Kazhdan--Lusztig polynomials of fan matroids using the generating function $\Psi(t,u)$ given in Theorem~\ref{thm-fan-pro}. To this end, we first provide the recurrence satisfied by $Q_{F_n}(t)$.

\begin{proposition}\label{prop-recu-fan}
The sequence $\{Q_{F_n}(t)\}_{n=0}^{\infty}$ satisfies the recurrence relation 
\begin{align*}
   4nt(t+4)Q_{F_{n}}(t) - 8ntQ_{F_{n+1}}(t) -  (n+3)(t+4)Q_{F_{n+2}}(t) + 2(n+3)Q_{F_{n+3}}(t) = 0
\end{align*}
for all $n\geq 0$, with initial conditions $Q_{F_{0}}(t)=1$, $Q_{F_{1}}(t)=1$,  and $Q_{F_{2}}(t)=2$.
\end{proposition}

\begin{proof}
By Theorem~\ref{thm-fan-pro}, the generating function $\Psi(t,u) = \sum_{n=0}^{\infty}Q_{F_{n}}(t)u^{n}$  admits the
closed form \eqref{fan-pro2}. A direct computation shows that $\Psi(t,u)$ satisfies the  first-order linear differential
equation with respect to $u$
\begin{align}\label{eq:pde-1}
\big((4t^2 +16t)u^3-8tu^2-(t+4)u+2\big) \Psi_u(t,u)+(8tu-t-4)\Psi(t,u)-6tu+t+2=0.
\end{align}
Substituting the power series expansion of $\Psi(t,u)$ into \eqref{eq:pde-1} yields
\begin{align}\label{eq:pde-2}
    \big((4t^2 +16t)u^3-8tu^2
    &-(t+4)u+2\big)\sum_{n=0}^{\infty}(n+1)Q_{F_{n+1}}(t)u^{n}
\nonumber 
\\
   & \quad +(8tu-t-4)\sum_{n=0}^{\infty}Q_{F_{n}}(t)u^{n}-6tu+t+2=0.
\end{align}
Since $F_0$ is the single vertex graph of rank 0, we have $Q_{F_0}(t)=1$ by definition. 
Extracting  the coefficients of $u^0$ and $u^1$ from 
\eqref{eq:pde-2},  we obtain 
$$2Q_{F_1}(t)-(t+4)Q_{F_0}(t)+t+2=0,$$
$$-2(t+4)Q_{F_1}(t)+4Q_{F_2}(t)+8tQ_{F_0}(t)-6t=0.$$
Solving these equations gives $Q_{F_1}(t)=1$ and
$Q_{F_2}(t)=2$.
For $n \ge 0$, extracting the coefficient of $u^{n+2}$ from \eqref{eq:pde-2} yields precisely the stated recurrence relation.
This completes the proof.
\end{proof}

Now we  prove  Theorem \ref{thm-fan}.

\begin{proof}[Proof of Theorem \ref{thm-fan}]
Fix $n\geq 1$ and define
$
f_n(t):=\sum_{k=0}^{\infty}a_{n,k}t^{k}$,
where 
\begin{align} \label{ank-new}
a_{n,k}:=\frac{2^{n-2k-1}(n-2k)}{n}\binom{n}{k}
\end{align}
for $0 \le k \le \lfloor (n-1)/2\rfloor$, 
and $a_{n,k} = 0$ otherwise.  For completeness, we also define $a_{n,k} = 0$ whenever $k < 0$.
Let $c_{n,k}$ denote the coefficient of $t^k$ in $Q_{F_n}(t)=\sum_{k=0}^{\infty}c_{n,k}t^k$, with the same convention that $c_{n,k}=0$ if $k<0$ or $n<2k+1$.
To prove Theorem~\ref{thm-fan}, it suffices to show that  the arrays $\{a_{n,k}\}$ and $\{c_{n,k}\}$ satisfy the same recurrence relation  together with the same  initial conditions.
By direct computation for $n = 1, 2, 3$, one readily verifies that $f_n(t) = Q_{F_n}(t)$  in these cases.

From Proposition~\ref{prop-recu-fan}, comparing coefficients of $t^{k+2}$ yields the recurrence
\begin{align*}
4n c_{n,k} + 16n c_{n,k+1} &- 8n c_{n+1,k+1} - (n+3)c_{n+2,k+1} \vspace{3mm}\\
& \qquad
- 4(n+3)c_{n+2,k+2} + 2(n+3)c_{n+3,k+2} = 0,
\end{align*}
which can be rewritten as
\begin{align}\label{an,k-rec1}
4n\bigl(c_{n,k} + 4 c_{n,k+1} - 2 c_{n+1,k+1}\bigr)
    - (n+3)\bigl(c_{n+2,k+1} + 4 c_{n+2,k+2} - 2 c_{n+3,k+2}\bigr) = 0.
\end{align}
It remains to verify that the array $\{a_{n,k}\}$ satisfies the same recurrence. 
Define
\begin{align*}
R(n,k):=a_{n,k} + 4 a_{n,k+1} - 2 a_{n+1,k+1}.
\end{align*}
We suffice to to prove that
\begin{align}\label{Rnk-rec}
    4n R(n,k)-(n+3)R(n+2,k+1)=0,  \qquad
    \text{for~all~} n\geq 1 \text{ and } k\in \mathbb{Z}.
\end{align}
For this purpose, we first show that, for all $n\ge 1$ and $k\in \mathbb{Z}$,
\begin{align}\label{Rnk-new}
R(n,k)=
\begin{cases}
a_{2k+1,k}, & \text{if } k\ge 0 \text{ and } n=2k+1,\\
0, & \text{otherwise}.
\end{cases}
\end{align}
We prove \eqref{Rnk-new} by considering the following five cases.

\smallskip
\noindent
\textit{Case 1: $k<0$ and $n\ge 1$.}
If $k\le -2$, then both $k$ and $k+1$ are negative. By definition,
$a_{n,k}=a_{n,k+1}=a_{n+1,k+1}=0$,
which implies  $R(n,k)=0$.
For $k=-1$, we have
$R(n,-1)=4a_{n,0}-2a_{n+1,0}$. Substituting $k=0$ into \eqref{ank-new}, we obtain
$a_{n,0}=2^{n-1}$ for all $n\geq 1$.
Consequently, $R(n,-1)=4\cdot 2^{n-1}-2\cdot 2^n =0$. 
Hence,
$
R(n,k)=0$ for all $n\ge 1$ and  $k<0$.

\smallskip
\noindent
\textit{Case 2: $k\ge 0$ and $1\le n\le 2k$.}
Since 
$
n<2k+1$, we have $n<2(k+1)+1$  and $n+1<2(k+1)+1
$.
Thus,
$
a_{n,k}=a_{n,k+1}=a_{n+1,k+1}=0,
$
which gives $R(n,k)=0$.

\smallskip
\noindent
\textit{Case 3: $k\ge 0$ and $n=2k+1$.}
We have
$$
a_{n,k}=a_{2k+1,k}=\frac{1}{2k+1}\binom{2k+1}{k}\neq 0.
$$
Moreover,
$
n<2(k+1)+1$
and
$n+1<2(k+1)+1$,
which implies
$
a_{n,k+1}=a_{n+1,k+1}=0$.
Hence,
$
R(n,k)=a_{2k+1,k}$.

\smallskip
\noindent
\textit{Case 4: $k\ge 0$ and $n=2k+2$.}
By definition, we have
$a_{n,k}\neq 0$, $a_{n,k+1}=0$, and $a_{n+1,k+1}\neq 0$.
Thus
$
R(2k+2,k)=a_{2k+2,k}-2a_{2k+3,k+1}.
$
A direct computation gives
\begin{align*}
    R(2k+2,k)
&=\frac{2}{k+1}\binom{2k+2}{k}
-\frac{2}{2k+3}\binom{2k+3}{k+1}\\
&=\frac{2}{k+1}\binom{2k+2}{k}
-\frac{2}{2k+3}\cdot \frac{2k+3}{k+1}\binom{2k+2}{k}
=0.
\end{align*}

\smallskip
\noindent
\textit{Case 5: $k\ge 0$ and $n\ge 2k+3$.}
In this case, all three terms $a_{n,k}$, $a_{n,k+1}$, and $a_{n+1,k+1}$ are nonzero. Using
$$
\binom{n}{k}=\frac{k+1}{n+1}\binom{n+1}{k+1},
\qquad
\binom{n}{k+1}=\frac{n-k}{n+1}\binom{n+1}{k+1},
$$
we compute
\begin{align*}
a_{n,k}+4a_{n,k+1}
&=\frac{2^{n-2k-1}(n-2k)}{n}\binom{n}{k}
 +\frac{2^{n-2k-1}(n-2k-2)}{n}\binom{n}{k+1}\\
&=\frac{(n-2k)(k+1)2^{n-2k-1}+(n-2k-2)(n-k)2^{n-2k-1}}{n(n+1)}
   \binom{n+1}{k+1}\\
&=\frac{(n-2k-1)2^{\,n-2k-1}}{n+1}\binom{n+1}{k+1}\\
&=2a_{n+1,k+1}.
\end{align*}
Hence $R(n,k)=0$.

\smallskip

Combining the above five cases, we obtain \eqref{Rnk-new}. 

We now prove \eqref{Rnk-rec}.
By \eqref{Rnk-new}, the left-hand side of \eqref{Rnk-rec} is zero unless $n=2k+1$.  Therefore, it remains to establish that
\begin{align}\label{R-2k+1-k-new}
    4(2k+1)a_{2k+1,k}-(2k+4)a_{2k+3,k+1}=0, \qquad
    \text{for~all~} k\geq 0.
\end{align}
Indeed,
\begin{align*}
    4(2k+1)a_{2k+1,k}-(2k+4)a_{2k+3,k+1}
    &=4\binom{2k+1}{k}-\frac{2k+4}{2k+3}\binom{2k+3}{k+1}\\
    &=4\binom{2k+1}{k}-\frac{2k+4}{2k+3}\cdot \frac{2(2k+3)}{k+2}\binom{2k+1}{k}=0.
\end{align*}
This proves \eqref{R-2k+1-k-new}, and the proof is complete.
\end{proof}

\section{\texorpdfstring{Inverse $Z$-polynomials of fan matroids}{Inverse Z-polynomials of fan matroids}}\label{In_Z_Fan}


In this section, we compute the inverse $Z$-polynomials of fan matroids 
using generating functions, following the approach used for inverse 
Kazhdan--Lusztig polynomials.
Define
\begin{align*}
    \Psi_{Y}(t,u):=\sum_{n=0}^{\infty}Y_{F_{n}}(t)u^{n}.
\end{align*}
We first derive an explicit expression for $ \Psi_{Y}(t,u)$.

\begin{theorem}\label{thm-invz-gen}
We have 
\begin{align}\label{YFtu}
    \Psi_{Y}(t,u)=\frac{2(-1+u+tu)}{-3+4(1+t)u+\sqrt{1-4tu^{2}}}.
\end{align}
\end{theorem}

The argument is parallel to that of \eqref{fan-pro2}.  We only state the
key lemma needed to pass to generating functions.
For any $A=(A_1,\dots,A_{2k})\in\mathcal{C}'_n$ with
$A_{2i-1}=(a_{2i-1})$ and
$A_{2i}=(b_{i1},\dots,b_{i\ell_i})\in\mathcal{S}_{a_{2i}}$
for $1\le i\le k$,
we define
\begin{equation*}
\tilde w(A):=
\begin{cases}
\displaystyle
\prod_{i=1}^{k}
\bigl(2^{\ell_i-1}t^{\ell_i}\bigr)\,Q_{F_{a_{2i-1}}}(t),
& \text{if} ~\ell_k\neq 0,
\\[8pt]
\displaystyle Q_{F_{a_{2k-1}}}(t)
\prod_{i=1}^{k-1}
\bigl(2^{\ell_i-1}t^{\ell_i}\bigr)\,Q_{F_{a_{2i-1}}}(t),
& \text{if~}\ell_k=0.
\end{cases}
\end{equation*}
As usual, we set the empty product equal to 1.

Parallel to Lemma \ref{wterm}, we  have the following result.

\begin{lemma}\label{invz-wterm}
For any $C \in \mathcal{C}(F_n)$, we have
\begin{equation}\label{invz-wei-new} (-1)^{\rk( F_n)} \cdot (-1)^{\rk( F_n[C])}Q_{F_n[C]}(t) \cdot t^{\rk( F_n/C)}\mu_{F_n/C} = \tilde{w}\left(\phi(C)\right). \end{equation}
\end{lemma}

\begin{proof}
Let $C\in\mathcal{C}(F_n)$ and write $\phi(C)=(A_1,\dots,A_{2k})$.
Combining \eqref{chim-chi} with \eqref{chi-Fcon-C}, we obtain
\begin{equation}\label{invz-wterm-1}
\chi_{\M(F_n/C)}(t)
=t^{-1}\chi_{F_n/C}(t)
=t^{-k}\prod_{i=1}^{k}\chi_{F_{\ell_i}}(t).
\end{equation}
Moreover, 
\begin{equation}\label{invz-wterm-2}
\chi_{F_n}(t)=
\begin{cases}
t, & n=0,\\[4pt]
t(t-1)(t-2)^{\,n-1}, & n\ge 1.
\end{cases}
\end{equation}
Recall that, for any matroid, the M\"obius invariant is equal to the constant term of its characteristic polynomial; see the proof of \cite[Proposition~2.1]{gao2025inverse}. 
Since $\ell_i=0$ if and only if $i=k$, 
it follows from \eqref{invz-wterm-1} and \eqref{invz-wterm-2} that
\begin{equation*}
\mu_{F_n/C}=
\begin{cases}
\displaystyle
\prod_{i=1}^{k-1}(-1)^{\ell_i}2^{\ell_i-1},
& \ell_k=0,\\[8pt]
\displaystyle
\prod_{i=1}^{k}(-1)^{\ell_i}2^{\ell_i-1},
& \ell_k\neq 0.
\end{cases}
\end{equation*}
Using $\rk(F_n/C)=\ell_1+\cdots+\ell_k$ together with 
\eqref{rkFnc} and \eqref{QFnc}, we complete the proof.
\end{proof}

Define
$$
\tilde\Phi(u):=\sum_{n\ge0}\left(\sum_{A\in\mathcal{C}'_n}\tilde w(A)\right)u^n.
$$
By \eqref{invz-def} and Lemma~\ref{invz-wterm}, and  since
$\phi:\mathcal{C}(F_n)\to\mathcal{C}'_n$ is a bijection, we have 
$$
Y_{F_n}(t)=\sum_{C\in\mathcal{C}(F_n)}\tilde w(\phi(C))
=\sum_{A\in\mathcal{C}'_n}\tilde w(A)$$ 
for each $n\ge0$.
Multiplying by $u^n$ and
summing over $n\ge0$ yields
\begin{equation}\label{eq:PsiY=tildePhi}
\Psi_Y(t,u)=\tilde\Phi(u).
\end{equation}

As in the computation of the generating function for inverse Kazhdan--Lusztig polynomials, we decompose $\tilde\Phi(u)$ into its odd and even parts.
For each $n\ge 0$,
a type $\tilde{\mathcal A}^o$ structure of size $n$ is the weak
composition $(n)$, with weight
$\tilde w^o((n)):=Q_{F_n}(t)$.
For $n\ge 1$,
a type $\tilde{\mathcal A}^e$ structure of size $n$ is a composition
$(b_1,\dots,b_k)\in\mathcal{S}_n$, with weight
$\tilde w^e((b_1,\dots,b_k)):=2^{k-1}t^{k}$. 
Besides, the unique $\tilde{\mathcal A}^e$ structures of size $0$ is $()$, which is weighted by $1$.
Define
$$
\tilde\Phi^o(u):=\sum_{n\ge1}Q_{F_n}(t)u^n,
\qquad
\tilde\Phi^e(u):=\sum_{n\ge1}\left(\sum_{(b_1,\dots,b_k)\in\mathcal{S}_n}
2^{k-1}t^{k}\right)u^n.
$$
 
Parallel to Lemma~\ref{gen-o and e}, we have the following result.

\begin{lemma}\label{Inv-Z-gen-o and e}
We have
\begin{equation}\label{invz-gf-1}
\tilde\Phi^o(u)=\Psi(t,u)-1,
\end{equation}
\begin{equation}\label{invz-gf-2}
\tilde\Phi^e(u)=\frac{tu}{1-u-2tu}.
\end{equation}
\end{lemma}

\begin{proof}
The identity \eqref{invz-gf-1}  follows directly from the definition of
$\Psi(t,u)$.
For \eqref{invz-gf-2}, a composition of length $k$ contributes weight
$2^{k-1}t^k$ and has generating function $(u/(1-u))^k$.  Summing over
$k\ge1$ gives
$$
\tilde\Phi^e(u)
=\sum_{k\ge1}2^{k-1}t^k\left(\frac{u}{1-u}\right)^k
=\frac{tu}{1-u-2tu},
$$
which completes the proof.
\end{proof}

Parallel to Lemma~\ref{invkl-Phi}, we have the following result.

\begin{lemma}\label{Inv-Z-Phi}
We have
\begin{equation}\label{invz-eq-another}
\Psi_{Y}(t,u)
=
\frac{(1+\tilde{\Phi}^o(u))(1+\tilde{\Phi}^e(u))}
     {1-\tilde{\Phi}^e(u)\tilde{\Phi}^o(u)}.
\end{equation}
\end{lemma}

\begin{proof}
By construction, $\Psi_Y(t,u)$ enumerates finite alternating
concatenations of an odd part and an even part, with multiplicative
weights.  Hence
\[
\Psi_Y(t,u)
=(1+\tilde\Phi^o(u))(1+\tilde\Phi^e(u))
\sum_{r\ge0}\bigl(\tilde\Phi^o(u)\tilde\Phi^e(u)\bigr)^r,
\]
which simplifies to \eqref{invz-eq-another}.
\end{proof}

We proceed to prove 
Theorem~\ref{thm-invz-gen}.

\begin{proof}[Proof of Theorem~\ref{thm-invz-gen}]
Substituting \eqref{invz-gf-1} and \eqref{invz-gf-2} into
\eqref{invz-eq-another}, and then using \eqref{fan-pro2} for $\Psi(t,u)$,
we obtain \eqref{YFtu}. This completes the proof.
\end{proof}


We are now ready to prove Theorem~\ref{thm-fan-invz}.
In contrast to the proof of Theorem~\ref{thm-fan},
we establish Theorem~\ref{thm-fan-invz}
by using the palindromicity of $Y_{F_n}(t)$ and the theory of formal power series,
rather than deriving a recurrence relation for $Y_{F_n}(t)$
from its generating function.

\begin{proof}[Proof of Theorem~\ref{thm-fan-invz}]
By \cite[Lemma~2.2]{gao2025inverse}, the polynomial $Y_{F_n}(t)$ is palindromic of degree $n$. It therefore admits a unique expansion in the basis
$\{t^j(1+t)^{n-2j}\}_{0\le j\le \lfloor n/2\rfloor}$, namely,
\begin{equation}\label{eq:pal-exp-Y}
Y_{F_n}(t)=\sum_{j=0}^{\lfloor n/2\rfloor} b_{n,j}\,t^j(1+t)^{n-2j}.
\end{equation}
Extracting the coefficient of $t^k$ from \eqref{eq:pal-exp-Y}, we obtain
\begin{align*}
    [t^k]Y_{F_n}(t)
    =\sum_{j=0}^{\min\{k,\lfloor n/2\rfloor\}} b_{n,j}\binom{n-2j}{k-j}=\sum_{j=0}^{\lfloor n/2\rfloor}b_{n,j}\binom{n-2j}{k-j},
\end{align*}
where the second equality holds since the binomial coefficient vanishes for $j>k$.
Hence,  \eqref{thm-fan-invz-equation} will follow
once the coefficients $b_{n,j}$ are determined. More precisely,
it suffices to show  that
\begin{align}\label{bnj-expr-ge1-new}
b_{n,j}=
\begin{cases}
2^{n-1}, & \text{if } j=0,\\[1mm]
\displaystyle
2^{n-2j-1}\sum_{r=1}^{j}\frac{(-1)^{r} r}{2j-r}\cdot\frac{2r+n-2j}{r+n-2j}
\binom{2j-r}{j}\binom{r+n-2j}{r},
& \text{if } 1\le j\le \lfloor n/2\rfloor,
\end{cases}
\end{align}
for all $n\ge 1$.

Multiplying both sides of \eqref{eq:pal-exp-Y} by $u^n$,
summing over all $n\ge 0$,
and applying Theorem~\ref{thm-invz-gen}, we obtain
\begin{align}\label{bnj-gen-ori}
    \sum_{n=0}^{\infty}\sum_{j=0}^{\lfloor n/2\rfloor} b_{n,j}\,t^j(1+t)^{n-2j}u^n=\frac{2(-1+u+tu)}{-3+4(1+t)u+\sqrt{1-4tu^{2}}}.
\end{align}
Now set
\begin{equation*}
w:=u(1+t),\qquad z:=\frac{t}{(1+t)^2}.
\end{equation*}
Then \eqref{bnj-gen-ori} becomes
\begin{align}\label{bnj-gen-new}
    \sum_{n=0}^{\infty}\sum_{j=0}^{\lfloor n/2\rfloor} b_{n,j}\,z^{j}w^{n} =\sum_{j=0}^{\infty}\sum_{n=2j}^{\infty} b_{n,j}\,w^{n}z^{j} =\frac{2(w-1)}{-3+4w+\sqrt{1-4zw^2}}.
\end{align}
Let
\begin{align*}
    C(x):=\sum_{i=0}^{\infty}C_i x^i=\frac{1-\sqrt{1-4x}}{2x}
\end{align*}
be the generating function for the Catalan numbers.
Then
\begin{align}\label{cata-gen-new}
    \sqrt{1-4zw^2}=1-2zw^2 C(zw^2).
\end{align}
Substituting \eqref{cata-gen-new} into \eqref{bnj-gen-new} yields
\begin{align}\label{bnj-fra-gem-new}
    \sum_{j=0}^{\infty}\sum_{n=2j}^{\infty} b_{n,j}\,w^{n}z^{j} = \frac{w-1}{2w-1-zw^2C(zw^2)}=\frac{1-w}{1-2w}\cdot \frac{1}{1-\frac{zw^2C(zw^2)}{2w-1}}.
\end{align}
Since 
\begin{align*}
\frac{zw^2C(zw^2)}{2w-1} \in \mathbb{R}[[w,z]]
\end{align*}
has vanishing constant term, the geometric series 
\begin{align}\label{gem-new}
    \sum_{r=0}^{\infty}\Big(\frac{zw^2C(zw^2)}{2w-1}\Big)^r=\frac{1}{1-\frac{zw^2C(zw^2)}{2w-1}}
\end{align}
is well-defined in $\mathbb{R}[[w,z]]$.  Substituting \eqref{gem-new} into \eqref{bnj-fra-gem-new} yields
\begin{align}\label{key-bnj-new}
\sum_{j=0}^{\infty}\sum_{n=2j}^{\infty} b_{n,j}\,w^{n}z^{j} =\frac{1-w}{1-2w}\sum_{r=0}^{\infty}\Big(\frac{zw^2C(zw^2)}{2w-1}\Big)^r.
\end{align}

Extracting the coefficients of $z^0$ from \eqref{key-bnj-new}, we obtain
\begin{align}\label{z0-coe-new}
    \sum_{n=0}^{\infty}b_{n,0}w^n=\frac{1-w}{1-2w}=(1-w)\sum_{n=0}^{\infty}2^n w^n.
\end{align}
It follows that
\begin{align}\label{bn0}
    b_{n,0}=2^{n-1}\qquad \text{for all } n\ge 1.
\end{align}
Now let $j\ge 1$. Extracting the coefficient of $z^j$ 
\eqref{key-bnj-new}  gives
\begin{align}\label{zj-1-new}
    \sum_{n=2j}^{\infty}b_{n,j}w^n =w^{2j}(w-1)\sum_{r=1}^{j}\frac{[x^{j-r}]C(x)^r}{(2w-1)^{r+1}}.
\end{align}
By~\cite[Lemma~2]{lang2000polynomials},
\begin{align}\label{cata}
    [x^k]C(x)^r = \frac{r}{2k+r}\binom{2k+r}{k},   \qquad \text{for all~} k\ge 0 \text{~and~} r\ge1.
\end{align}
Substituting \eqref{cata}  into \eqref{zj-1-new}, we obtain
\begin{align}\label{bnj-expand-new}
    \sum_{n=2j}^{\infty}b_{n,j}w^n= w^{2j}\sum_{r=1}^{j}\frac{r}{2j-r}\binom{2j-r}{j}\frac{w-1}{(2w-1)^{r+1}}.
\end{align}
Extracting coefficients of $w^n$ from \eqref{bnj-expand-new} yields
\begin{align}\label{bnj-simp-new}
    b_{n,j}=\sum_{r=1}^{j}\Big([w^{n-2j}]\frac{w-1}{(2w-1)^{r+1}}\Big)\cdot \frac{r}{2j-r}\binom{2j-r}{j}.
\end{align}

Since 
\begin{align*}
    \frac{1}{(2w-1)^{r+1}}=(-1)^{r+1}\sum_{i=0}^{\infty}\binom{-r-1}{i}(-2w)^i=(-1)^{r+1}\sum_{i=0}^{\infty}\binom{r+i}{i}2^i w^i,
\end{align*}
it follows that
\begin{align*}
[w^{i}]\frac{w-1}{(2w-1)^{r+1}}
=
\begin{cases}
(-1)^r, & i=0, \\[1.2ex]
(-1)^{r}2^{i-1}\dfrac{2r+i}{r+i}\dbinom{r+i}{r}, & i\ge 1.
\end{cases}
\end{align*}
Moreover, when  $r\ge 1$, the case $i=0$ is also covered by the formula
for  $i\ge 1$.
Hence,  for all $n\ge 2j$, 
\begin{align}\label{bnj-simp-new-2}
    [w^{n-2j}]\frac{w-1}{(2w-1)^{r+1}}=(-1)^{r}2^{n-2j-1}\frac{2r+n-2j}{r+n-2j}\binom{r+n-2j}{r}.
\end{align}
Substituting \eqref{bnj-simp-new-2} into \eqref{bnj-simp-new} yields
\begin{align*}
    b_{n,j}= 2^{n-2j-1}\sum_{r=1}^{j}\frac{(-1)^{r} r}{2j-r}\cdot\frac{2r+n-2j}{r+n-2j}
\binom{2j-r}{j}\binom{r+n-2j}{r},   \qquad \text{for } 1\le j\le \lfloor n/2\rfloor.
\end{align*}
Together  with \eqref{bn0}, we obtain \eqref{bnj-expr-ge1-new}. The proof is complete.
\end{proof}


%

\section{The deletion formula of Braden, Ferroni, Matherne, and Nepal}\label{dele-section}

In this section, we give an alternative proof of Theorems \ref{thm-fan-pro} and \ref{thm-invz-gen} using the deletion formulas for inverse Kazhdan--Lusztig polynomials and inverse $Z$-polynomials due to Braden, Ferroni, Matherne, and Nepal \cite{braden2025deletion}.

Throughout this section,
we view flats as subsets of $E(G)$,
rather than as the corresponding
partitions of $V(G)$.
For any subset $A \subseteq E(G)$, the closure of $A$ is defined by
\begin{align*}
\mathrm{cl}(A)
=
\{(x,y)\in E(G)
: x \text{ and } y
\text{ are connected  in } G[A]\}.
\end{align*} 
A subset $F \subseteq E(G)$
is a flat if $\mathrm{cl}(F) = F$.
Given a graph $G$,
let $\mathcal{L}(G)$ denote
the lattice of flats
of the graphic matroid $\M(G)$.
For an edge $e\in E(G)$, define
\begin{equation*}
 \mathcal{T}_e(G) := \left\{ F \in \mathcal{L}(G): e\in F \text{ and } F\smallsetminus\{e\}\notin \mathcal{L}(G)\right\}.
\end{equation*}
We now define
$\tau(G)$ by
\begin{align*}
 \tau(G) := 
 \begin{cases}
 [t^{(\rk(G)-1)/2}] P_{G}(t), & \text{if $\rk(G)$ is odd},\\[6pt]
 0, & \text{if $\rk(G) $ is even}.
 \end{cases}
\end{align*}

Recall that $F_n$ denotes the fan graph with vertex set
$\{0,1,\dots,n\}$ and edge set
\begin{align*}
    E(F_n)=\{(0,i):1\le i\le n\}\ \cup\ \{(i,i+1):1\le i\le n-1\}.
\end{align*} 
We begin with the following lemma.

\begin{lemma}\label{tau-not-0}
Let  $e=(n-1,n)\in E(F_n)$ and let $F\in\mathcal{T}_e(F_n)$. Then 
\begin{align*}
 \tau(F_n[F]/e)\neq 0
 \quad
\text{if~and~only~if}
 \quad
 F_n[F]=F_{n-m+1},
\end{align*}
for some $m\in\{1,\dots,n-1\}$ such that $n-m$ is odd.
\end{lemma}

\begin{proof}
We first show that
for any $F\in\mathcal{T}_e(F_n)$,
if $F_n[F]$ is disconnected,
then $\tau(F_n[F]/e)=0$.
Indeed, if $F_n[F]$ is disconnected,
then its graphic matroid
is a direct sum of nonempty matroids.
The same holds for $F_n[F]/e$.
By \cite[Lemma~2.7]{braden2020kazhdan},
we have $\tau(F_n[F]/e)=0$.
Thus we only need to consider
$F\in\mathcal{T}_e(F_n)$
for which $F_n[F]$ is connected.

By the definition of $\mathcal{T}_e(F_n)$,
$F$ is a flat
but $F\setminus\{e\}$ is not.
Combined with the monotonicity of the closure operator, 
this implies that
$e\in\mathrm{cl}(F\setminus\{e\})$.
Hence vertices $n-1$ and $n$
are connected by a path
in $F_n[F\setminus\{e\}]$.
Any such path avoiding $e$
must pass through the edge $(0,n)$,
so $(0,n)\in F$.
Since $F$ is a flat containing both $e$ and $(0,n)$, it follows that $(0,n-1)\in F$.
Otherwise $(0,n-1) \in \mathrm{cl}(F)$, contradicting that $F$ is a flat.
More generally,
for each $i\in\{2,\dots,n-1\}$,
consider the triangle with vertex set $\{0,i-1,i\}$ and edge set
$$
\{(0,i-1),(i-1,i),(0,i)\}.
$$
Assume that $(0,i)\in F$. If either $(0,i-1)$ or $(i-1,i)$ also belongs to $F$, then so does the other, by the same argument as above. In other words, whenever $(0,i)\in F$, the edges $(0,i-1)$ and $(i-1,i)$ either both belong to $F$ or both do not.
Since $\{(0,n-1),(n-1,n),(0,n)\} \subseteq F$, repeating this argument backwards along the fan graph, we see that, under the assumption that $F_n[F]$ is connected, the edges of $F_n[F]$ are arranged in consecutive fan blocks.
More precisely, there are only two possibilities for $F_n[F]$: either $F_n[F]$ is a fan graph $F_{n-m+1}$ for some $m\in\{1,\dots,n-1\}$, or $F_n[F]$ is obtained by identifying several fan graphs at $0$ and contains the edge set $\{(0,n-1),(n-1,n),(0,n)\}$.

To prove this lemma,
we now  suppose that $\tau(F_n[F]/e)\neq 0$. If $F_n[F]$ is formed
by identifying several fan graphs at $0$, then its graphic matroid
is again a direct sum of nonempty matroids. By \cite[Lemma~2.7]{braden2020kazhdan},
we have $\tau(F_n[F]/e)=0$. 
Therefore, $F_n[F]=F_{n-m+1}$
for some $m\in\{1,\dots,n-1\}$. Moreover, recall that $\tau(G)=0$ whenever $\M(G)$ has even rank.
Since $\rk(F_{n-m+1}/e)=n-m$, it follows that $m$ must be such that $n-m$ is odd.
Conversely, assume that
$$
F_n[F]=F_{n-m+1}
\qquad\text{and}\qquad
n-m=2j+1
$$
for some $0\le j\le\lfloor (n-2)/2\rfloor$.
Observe that the graph $F_{n-m+1}/e$
contains parallel edges, and $\operatorname{si}(F_{n-m+1}/e) \cong F_{n-m}$.
Since the Kazhdan--Lusztig polynomial is
also determined by the lattice of flats,
we have
$$P_{F_{n-m+1}/e}(t)=P_{F_{n-m}}(t)=P_{F_{2j+1}}(t).$$
In particular, $\tau(F_{n-m+1}/e)=\tau(F_{2j+1})$.
Hence, by \cite[Theorem~1.1]{xie2018},
\begin{equation}\label{flat-F-n-de-2} 
\tau(F_n[F]/e)=\tau(F_{2j+1})=\frac{1}{j+1}\binom{2j}{j}=C_j \neq 0. \end{equation}
This completes the proof.
\end{proof}

Now we present an alternative proof of Theorem~\ref{thm-fan-pro}.

\begin{proof}[Second proof of Theorem~\ref{thm-fan-pro}]
Let $\M(G)$ be a graphic matroid,
and let $e$ be a non-coloop element of $\M(G)$.
By \cite[Theorem~1.4]{braden2025deletion}, we have
\begin{equation}\label{dele-formula}
Q_{G}(t)
=
Q_{G\smallsetminus e}(t)
+
(1+t)Q_{G/e}(t)
-
\sum_{F\in\mathcal{T}_e(G)}
\tau(G[F]/e)
t^{\frac{\rk(F)}{2}}
Q_{G/F}(t).
\end{equation}
In a graphic matroid, an edge is a coloop if and only if it is a bridge.
Let $e=(n-1,n)\in E(F_n)$. Clearly, $e$ is not a coloop of $F_n$. 
We apply \eqref{dele-formula} with $G=F_n$ and this element $e$.

Observe that
$F_n\smallsetminus e \cong F_{n-1}\oplus B_1$,
since $(0,n)$ is a bridge (coloop) in $F_n\smallsetminus e$.
Moreover, $F_n/e$ contains parallel edges,
and $\operatorname{si}(F_n/e)\cong F_{n-1}$,
so $Q_{F_n/e}(t)=Q_{F_{n-1}}(t)$.
By multiplicativity of $Q_{\M}(t)$
under direct sums \cite[Lemma~3.1]{alice2020inverseKL}
and $Q_{B_1}(t)=1$, we obtain
$$Q_{F_n\smallsetminus e}(t)=Q_{F_{n-1}}(t).$$
Thus
\begin{equation*}
    Q_{F_n}(t)=(t+2)Q_{F_{n-1}}(t)- \sum_{F\in \mathcal{T}_e(F_n)} \tau(F_{n}[F]/e)t^{\frac{\rk(F)}{2}} Q_{F_{n}/F}(t).
\end{equation*}
By Lemma~\ref{tau-not-0},
nonzero contributions come exactly from
flats $F$ such that $F_n[F]=F_{n-m+1}$
for some $m\in\{1,\dots,n-1\}$
with $n-m$ odd.
Write $n-m=2j+1$,
where $0\le j\le\lfloor (n-2)/2\rfloor$.
Then $\rk(F)=2j+2$,
so $t^{\rk(F)/2}=t^{j+1}$,
and
\begin{equation}\label{flat-F-n-de}
\operatorname{si}(F_n/F)\cong F_{m-1}=F_{n-2j-2}.
\end{equation}
Contracting $e$ in $F_n[F]$ yields
(up to simplification) the fan matroid $F_{n-m}=F_{2j+1}$,
so
\[
\tau(F_n[F]/e)=\tau(F_{2j+1}),
\]
which is given by \eqref{flat-F-n-de-2}. Hence
\begin{equation}\label{iKL-de-1}
Q_{F_n}(t)
=
(t+2)Q_{F_{n-1}}(t)
-
\sum_{j=0}^{\lfloor (n-2)/2 \rfloor}
C_{j}t^{j+1} Q_{F_{n-2j-2}}(t).
\end{equation}

Multiply both sides by $u^{n}$
and sum over all $n\ge 2$. This gives
\begin{equation}\label{del-qn-rec}
    \Psi(t,u)-1-u=u(t+2)(\Psi(t,u)-1)-tu^{2}\Psi(t,u)\sum_{j=0}^{\infty}C_{j}t^{j}u^{2j}.
\end{equation}
Using the generating function
\begin{align}\label{gene-cata}
\sum_{j=0}^{\infty}C_{j}x^{j}=\frac{1-\sqrt{1-4x}}{2x},
\end{align}
equation \eqref{del-qn-rec} simplifies to
\begin{align*}
    \Psi(t,u)=\frac{2(1-(1+t)u)}{3-(4+2t)u-\sqrt{1-4tu^{2}}}.
\end{align*}
Multiplying the numerator and denominator by the conjugate of the denominator, we thus derive the desired formula.
\end{proof}

We also present an alternative proof of Theorem~\ref{thm-invz-gen},
analogous to that of Theorem~\ref{thm-fan-pro}.

\begin{proof}[Second proof of Theorem~\ref{thm-invz-gen}]

By \cite[Theorem~1.4]{braden2025deletion},
for any graphic matroid $\M(G)$
and any non-coloop element $e$ of $\M(G)$,
\begin{equation}\label{eq:Y-deletion}
Y_{G}(t)=Y_{G\smallsetminus e}(t)+(1+t)Y_{G/e}(t)
-\sum_{F\in\mathcal{T}_e(G)}
\tau(G[F]/e)t^{\frac{\rk(F)}{2}}Y_{G/F}(t).
\end{equation}
Let $e=(n-1,n)\in E(F_n)$.
We apply \eqref{eq:Y-deletion} with $G=F_n$ and this element $e$.
Recall from the proof of Theorem~\ref{thm-fan-pro} that
$F_n\smallsetminus e\cong F_{n-1}\oplus B_1$
and $\operatorname{si}(F_n/e)\cong F_{n-1}$.
Thus $Y_{F_n/e}(t)=Y_{F_{n-1}}(t)$.
By multiplicativity of $Y_{\M}(t)$ under direct sums
and $Y_{B_1}(t)=t+1$, we obtain
$$
Y_{F_n\smallsetminus e}(t)=(t+1)Y_{F_{n-1}}(t).
$$
Substituting these into \eqref{eq:Y-deletion} gives
\begin{equation}\label{eq:YFn-rec-raw}
Y_{F_n}(t)=2(t+1)Y_{F_{n-1}}(t)
-\sum_{F\in\mathcal{T}_e(F_n)}
\tau(F_n[F]/e)t^{\frac{\rk(F)}{2}}Y_{F_n/F}(t).
\end{equation}
From Lemma~\ref{tau-not-0},
nonzero contributions to the sum
come from flats $F$ such that $F_n[F]=F_{n-m+1}$, where 
$m\in\{1,\dots,n-1\}$ and $n-m$ is odd.
Substituting the corresponding results
\eqref{flat-F-n-de} for $\operatorname{si}(F_n/F)$
and \eqref{flat-F-n-de-2} for $\tau(F_n[F]/e)$
into \eqref{eq:YFn-rec-raw} yields
\begin{equation}\label{eq:YFn-rec}
Y_{F_n}(t)
=2(t+1)Y_{F_{n-1}}(t)
-\sum_{j=0}^{\lfloor (n-2)/2\rfloor} C_j\,t^{j+1}\,Y_{F_{n-2j-2}}(t),
\qquad \text{for}~n\ge 2.
\end{equation}

Multiply both sides of \eqref{eq:YFn-rec} by $u^n$
and sum over all $n\ge2$.
The left-hand side becomes
$\Psi_Y(t,u)-Y_{F_0}(t)-Y_{F_1}(t)u$.
Using $Y_{F_0}(t)=1$ and $Y_{F_1}(t)=t+1$,
we obtain
\begin{equation}\label{eq:PsiY-func}
\Psi_Y(t,u)-1-(t+1)u
=2u(t+1)\bigl(\Psi_Y(t,u)-1\bigr)
-tu^2\Psi_Y(t,u)\sum_{j\ge0} C_j(tu^2)^j .
\end{equation}
Applying the Catalan generating function \eqref{gene-cata},
equation \eqref{eq:PsiY-func} simplifies to the desired formula \eqref{YFtu}.
\end{proof}

\section{Log-concavity and real-rootedness}\label{sec-thm1.3}

This section is devoted to the proof of Theorem~\ref{thm-fan-log-concave}. 
The first approach proceeds by direct computation using the explicit formulas 
for the inverse Kazhdan--Lusztig polynomials of fan matroids, 
while the second approach is based on Newton's inequalities and multiplier sequences.

\begin{proof}[First proof of Theorem~\ref{thm-fan-log-concave}.]
Fix an integer $n \ge 1$.
For $0 \le k \le \lfloor (n-1)/2 \rfloor$, recall from
Theorem~\ref{thm-fan} that
\begin{align*} c_{n,k}=\frac{(n-2k)2^{n-2k-1}}{n}\binom{n}{k}.
\end{align*}
Clearly, $c_{n,k}>0$, i.e., the sequence $\{c_{n,k}\}_{k=0}^{\lfloor \frac{n-1}{2} \rfloor}$ has no internal zeros. To prove the log-concavity of this sequence, it suffices to prove that
\begin{align*}
    \quad c^{2}_{n,k} \ge c_{n,k-1}c_{n,k+1} \qquad \text{for~} 1 \le k \le \lfloor (n-1)/2 \rfloor-1.
\end{align*}
A straightforward computation yields
\begin{align*}
\frac{c^{2}_{n,k}}{c_{n,k-1}c_{n,k+1}}=\frac{(n-2k)^2}{(n-2k+2)(n-2k-2)} \cdot \frac{{\binom{n}{k}}^2}{\binom{n}{k+1}\binom{n}{k-1}}.
\end{align*}
For $1 \le k \le \lfloor (n-1)/2 \rfloor-1$, we have
$(n-2k+2)(n-2k-2)>0$ and $(n-2k)^{2}-(n-2k+2)(n-2k-2)=4>0$, which implies $$\frac{(n-2k)^2}{(n-2k+2)(n-2k-2)} \ge 1.$$
Moreover, it is well known that the sequence
$\{\binom{n}{k}\}_{k\ge0}$ is log-concave, that is,
\begin{align*}
{\binom{n}{k}}^2 \ge \binom{n}{k+1}\binom{n}{k-1}.
\end{align*}
 Combining these two inequalities, we conclude that
$$\frac{c^{2}_{n,k}}{c_{n,k-1}c_{n,k+1}}\ge 1,$$
which completes the proof.
\end{proof}

Let $\M$ be a matroid and write 
$Q_{\M}(t) = \sum_{i=0}^s c_i t^i$, where $s = \deg Q_{\M}(t)$. 
Define the normalization of $Q_{\M}(t)$ by
$$\mathscr{B}(Q_{\M}(t)):=\sum_{i=0}^s\binom{s}{i}c_it^i.$$
Equivalently, $\mathscr{B}(Q_{\M}(t))$ is the Hadamard product of 
$Q_{\M}(t)$ and $(1+t)^s$. 
Xie and Zhang \cite{XieZhang2025Logconcavity} conjectured that 
for every matroid $\M$, the polynomial $\mathscr{B}(Q_{\M}(t))$ 
is real-rooted.
If the conjecture holds, then by the Newton inequalities,
the coefficients of $\mathscr{B}(Q_{\M}(t))$  are ultra log-concave. This implies that
the coefficients of $Q_\M(t)$ are log-concave.

Recently, Braden, Ferroni, Matherne, and Nepal 
\cite{braden2025deletion} constructed a matroid of rank $19$ on $21$ elements 
that disproves Xie and Zhang's conjecture. 
We show that the conjecture holds for fan matroids.

\begin{theorem}\label{thm-fan-Blog}
For any positive integer $n$, the polynomial $\mathscr{B}(Q_{F_n}(t))$ has only real roots.
\end{theorem}

\begin{proof}
Fix an integer $n \ge 1$ and let $s=\lfloor (n-1)/2 \rfloor$. By Theorem \ref{thm-fan}, we have
\begin{align*}
    \mathscr{B}(Q_{F_{n}}(t))= \frac{2^{n-1}}{n}\sum_{k=0}^{s}(n-2k)\binom{s}{k}\binom{n}{k}\left(\frac{t}{4}\right)^{k}.
\end{align*}
Setting $x=t/4$, it suffices to show that the polynomial
\begin{align*}
    g_{n}(x):=\sum_{k=0}^{s}(n-2k)\binom{s}{k}\binom{n}{k}x^{k}
\end{align*}
is real-rooted. Recall that a sequence $\Gamma = \{\gamma_k\}_{k=0}^{\infty}$ of real numbers is called 
a multiplier sequence if, for every real polynomial
$
    f(t) = \sum_{k=0}^{n} a_k t^k
$
with only real zeros, the polynomial
\begin{equation*}
    \Gamma[f(t)] := \sum_{k=0}^{n} \gamma_k a_k t^k
\end{equation*}
also has only real zeros.  
Now,
\begin{align*}
    \sum_{k=0}^{n}(n-2k)\binom{n}{k}x^k=n(1-x)(1+x)^{n-1}
\end{align*}
has only real zeros. Moreover, ${\binom{s}{k}}_{k\ge 0}$ is a multiplier sequence
by \cite[Lemma~2.5]{zhang2018local}.
It follows that $g_n(x)$ is real-rooted. This completes the proof.
\end{proof}

\begin{proof}[Second proof of Theorem~\ref{thm-fan-log-concave}]
The assertion follows
from Theorem~\ref{thm-fan-Blog}
by Newton's inequalities
for real-rooted polynomials.
\end{proof}

\vspace{4mm}
\noindent
\textbf{Acknowledgements.} The first author is supported by the National Science Foundation for Post-doctoral Scientists
of China (No.2020M683544) and Guangdong Basic and Applied Basic Research Foundation (No.2024A1515011276).

\end{document}